\documentclass[a4paper,1pt]{article}

\RequirePackage[colorlinks,citecolor = magenta,urlcolor = red, linkcolor = blue]{hyperref}

\usepackage{amsmath, amssymb, amsthm, amsfonts, cases}
\usepackage{mathrsfs}
\usepackage{url}
\usepackage{authblk}
\usepackage[usenames]{color}
\usepackage{float}
\usepackage{subfigure}
\usepackage{caption2}
\usepackage{geometry}
\usepackage{multirow}
\usepackage{graphics}
\usepackage{epsfig}

\usepackage{epstopdf}
\allowdisplaybreaks

\newcommand{\LL} {\mathscr{L^*}}

\newcommand{\LLinf} {\mathscr{L}}

\newtheorem{proposition}{Proposition}

\newtheorem{remark}{Remark}

\begin{document}
	\title{Computing wedge probabilities: finite time horizon case}
	\author{Dmitry Muravey\footnote{e-mail:d.muravey87@gmail.com.}}
	\date{}
	\maketitle
		
\begin{abstract}  
We present an alternative to the well-known Anderson's formula for the probability that a first exit time from the planar region between two slopping lines $-a_1 t -b_1$ and $a_2 t + b_2 $ by a standard Brownian motion is greater than $T$. As the Anderson's formula, our representation is an infinite series from special functions. We show that convergence rate of both formulas depends only on terms $(a_1 + a_2)(b_1 + b_2)$ and $(b_1 + b_2)^2 /T$ and deduce simple rules of appropriate representation's choose. We prove that for any given set of parameters $a_1$, $b_1$, $a_2$, $b_2$, $T$ the sum of first 6 terms ensures precision $10^{-16}$. 
\end{abstract}

\tableofcontents

\section{Introduction}
Let $W = \{ W_t, \, t \geq 0 \}$ be a standard Brownian motion.  We consider the following exit times 
\[
T^s = \inf\{t\geq 0, \, X_t \notin [-a_1, a_2]\}, \quad 
T^w = \inf\{t\geq 0, \, X_t \notin [-a_1t - b_1, a_2t + b_2]\},
\]
where process $X$ is defined as $X = \{x_0 + W_t, \, t \geq 0\}$. Let's denote by $\mathbb{P}_{x_0}$ probability conditional on the process $X_t$ started at $X_0 = x_0$. We omit subscript $x_0$ if process $X_t$ starts at zero, i.e. if $X_t$ is a standard Brownian motion.  The probability $\mathbb{P} (T^w > T)$ is called wedge probability and have been studied by many authors \cite{Anderson}, \cite{BarbaEscriba}. Anderson \cite{Anderson} was the first to obtain the explicit representation for the $\mathbb{P} (T^w > T)$ in terms of infinite series from normal c.d.f. functions. The explicit formula for $\mathbb{P}(T^w = \infty)$  was found by Doob \cite{Doob}. Based on properties of Jacobi theta functions, Ycart and Drouilhet \cite{Ycart} have found alternative to Doob's formula. They also computed uniform precision estimates and proposed efficient numerical algorithm. Let us mention that connection of $\mathbb{P}(T^w = \infty)$ with theta functions has been pointed out by Salminen and Yor \cite{SalminenYor}.    

In this paper we generalize results from \cite{Ycart} to the finite time horizon case.  Based on Lie symmetries for the heat equation, we deduce alternative to Anderson's formula.  Our representation is an infinite series from Error functions with complex argument.  For numerical computations we derive another representation containing only real functions. Combining our results with Anderson' formula we propose simple numerical algorithm for computation of $\mathbb{P} (T^w > T)$.  We show that convergence rate of these two series depends only on two terms: $(a_1 + a_2)(b_1 + b_2)$ and $(b_1 + b_2)^2 / T$. Anderson' formula converges fast if at least one term is relatively high, our formula has fast convergence for the opposite case.  The algorithmic consequence is that computing at most six terms of the series either in Anderson's formula or in the new alternative suffices to approximate $\mathbb{P}_{x_0} (T^w > T)$ with precision smaller than $10^{-16}$. 

The rest of paper is organized by the following scheme: in Section \ref{sec:PDE} we show connections of $T^s$ with $T^w$ in terms of corresponded PDE boundary-value problems. Based on Lie symmetries for the heat equation we derive the solutions of these PDEs in explicit form. Section \ref{sec:Wedge_prob} contains explicit formulas for probability $\mathbb{P}(T^w > T)$. We re-derive famous Anderson's formula and present two new representations. Also we show that Doob's and Ycart and Drouilhet' formulas can be derived as the limiting form $T \rightarrow \infty$. In the last Section \ref{sec:numerics} we present upper bounds for the remainders of infinite series from Anderson' formula and its alternative. Based on these results, we deduce rules describing in which cases Anderson's formula or its alternative should be used and propose simple numerical algorithm for computation of wedge probability $\mathbb{P}(T^w > T)$.    

\section{PDE and Lie symmetries approach to analysis of stopping times $T^s$ and $T^w$}
\label{sec:PDE}
From standard results in probability theory $\mathbb{P}_{x_0} (T^s > T)$ and $\mathbb{P}_{x_0} (T^w > T)$ can be represented in the following form
\[
\mathbb{P}_{x_0} (T^s > T) = \int_{-a_1}^{a_2} u_{x_0}^{s} (x,T) dx, \quad 
\mathbb{P}_{x_0} (T^w > T) = \int_{-a_1 T -b_1}^{a_2 T + b_2} u_{x_0}^{w} (x,T) dx.
\]
Functions $u_{x_0}^{s} (x,t)$ and $u_{x_0}^{w} (x,t)$ solve the following Cauchy problems with killed boundary conditions for Fokker--Planck--Kolmogorov equation: 
\begin{equation}
\left\{ \begin{array}{l} 
\left(\LL - \partial / \partial_t \right) u_{x_0}^{s} = 0, \\
u_{x_0}^{s} (-a_1, t) = 0, \\
u_{x_0}^{s}(a_2, t) = 0, \\
u_{x_0}^{s}(x, 0) = \delta(x-x_0).
\end{array} \right.
\quad\quad\quad\quad 
\left\{ \begin{array}{l} 
\left(\LL - \partial / \partial_t \right) u_{x_0}^{w} = 0, \\
u_{x_0}^{w} (-a_1 t -b_1, t) = 0, \\
u_{x_0}^{w}(a_2 t + b_2, t) = 0, \\
u_{x_0}^{w}(x, 0) = \delta(x-x_0).
\end{array} \right.
\end{equation}
Here $\LL$ is the adjoint of $\LLinf$, which is the infinitesimal generator of the process $X_t$. In case of standard Brownian motion operator $\LLinf$ is self-adjoint, i.e. 
\[
\LL = \LLinf =\frac{1}{2}\frac{\partial^2 }{\partial x^2}.
\]
Initial condition $\delta(x - x_0)$ is a Dirac measure at the point $x_0$. In the next Proposition we present explicit formulas for $u_{x_0}^s$.
\begin{proposition}
\label{prop:u_s}
Function $u_{x_0}^{s} (x,t)$ has two equivalent representations ($l = a_1 + a_2$): 
\begin{eqnarray}
\label{eq:u_s_CFS}
 \begin{array}{l} 
u_{x_0}^s (x,t) = \frac{2}{l} \sum_{n = 0}^{\infty} \sin\left(\frac{n \pi}{l} (x + a_1)\right) 
sin\left(\frac{n\pi}{l} (x_0 + a_1) \right) e^{-\frac{n^2\pi^2}{2 l^2} t}
\\ 
u_{x_0}^{s} (x,t) = \frac{1}{\sqrt{2\pi t}}\sum_{n = -\infty}^{+\infty} \left[ 
e^{-\frac{(x-x_0 + 2nl)^2}{2t}} -
e^{-\frac{(x + x_0 + 2a_1+ 2nl)^2}{2t}} \right].
\end{array}
\end{eqnarray}
\end{proposition}

\begin{proof}
First formula can be derived from the following well-known representation
\[
u_{x_0}^{s} (x,t) = \sum_{n=1}^{\infty} A_n \sin\left( \omega_n (x+ \alpha) \right) e^{-\frac{\omega_n^2 t}{2}}.
\]
The phase shift $\alpha$ and frequencies $\omega_n$ are set to satisfy boundary conditions 	
$u_{x_0}^{s}(-a_1,t) = u_{x_0}^{s}(a_2,t) = 0$, i.e, $\alpha = a_1$, and $\omega_n = \pi n/l$. Hence at the moment $t = 0$ function $u_{x_0}^{s}(x,0)$ is Fourier series on the segment $[-a_1, a_2]$. Calculation of coefficients $A_n$ turns out to the main formula (\ref{eq:u_s_CFS}). 
\par 
Second formula in (\ref{eq:u_s_CFS}) can be obtained by using of Laplace transform with respect to the time variable $t$.  This integral transform reduces original initial boundary problem to the simple boundary problem for linear ODE:
\begin{eqnarray}
\label{eq:ODE_for_u_two-sided_Laplace_transform} 
\left \{ \begin{array}{l}
\frac{1}{2} \frac{d^2 w}{dx^2} - \zeta w = -\delta(x-x_0), 
\\ 
w(-a_1; \zeta) = 0, 
\\ 
w(a_2; \zeta) = 0.
\end{array}
\right.
\end{eqnarray}
Here $w(x; \zeta)$ is an image of transformation, i.e.
\[
w(x; \zeta) = \int_{0}^{\infty} e^{-\zeta t} u_{x_0}^{s} (x,t) dt, \quad\quad
u_{x_0}^{s}(x,t) = \frac{1}{2\pi i}\int_{\gamma-i\infty}^{\gamma+i\infty}e^{\zeta t} w(x ;\zeta) d\zeta. 
\]
The problem (\ref{eq:ODE_for_u_two-sided_Laplace_transform}) can be easily solved by a standard techniques.  Inversion of Laplace transform yields second formula (\ref{eq:u_s_CFS}). 
\end{proof}
\begin{remark}
We can also derive identity between two representations from (\ref{eq:u_s_CFS}) by using Poisson' summation formula
	\begin{equation}
\nonumber
\label{eq: Poisson summation formula}
\sum_{n = -\infty}^{+\infty} e^{-\frac{\pi^2 n^2}{2u}}\cos\left(\pi n v/u\right) = \sqrt{\frac{2u}{\pi}} e^{-\frac{v^2}{2u}}  \sum_{n = -\infty}^{+\infty} e^{-2n^2u}\cosh\left( 2nv\right).
\end{equation}
\end{remark}	
In he next proposition we establish connection of $u_{x_0}^w$ with $u_{x_0}^s$. 
\begin{proposition}
\label{prop:u_w}
Function $u_{x_0}^{w}$ can be explicitly represented in terms of $u_{x_0}^s$: 
\begin{equation}
\label{eq:u_w_from_u_s}
u_{x_0}^{w} (x,t) = 
\frac{a_{+} e^{\frac{(a_{+}x_0 + d)^2}{2a_{+} b_{+}} - \frac{(a_{+} x - d)^2}{2a_{+} (a_{+}t + b_{+})}} }{\sqrt{b_{+}} \sqrt{a_{+} t + b_{+}} }
u_{(a_{+}x_0-d) / b_{+}}^s \left( \frac{a_{+} x - d}{a_{+}t +b_{+}}, \, \frac{1}{b_{+} / a_{+}} - \frac{1}{t + b_{+} / a_{+}}\right).
\end{equation}
here constants $a_{+}$, $b_{+}$ and $d$ are equal 
\begin{equation}
\label{eq:a_b_d}
a_{+} = \frac{a_1 + a_2}{2}, 
\quad b_{+} = \frac{b_1 + b_2}{2}, 
\quad d = \frac{a_1 b_2 - a_2 b_1}{2}.
\end{equation}
\end{proposition}
\begin{proof}
   This fact follows from Lie symmetries for the heat equation. One can check by the direct calculations that function $u_{x_0}^{w}$ from (\ref{eq:u_w_from_u_s}) solves equation
	$\left(\LL - \partial_t\right) u_{x_0}^{w} = 0$. Note that function $u_{x_0}^s$ is equal to zero if $x = -a_1$ or $x = a_2$, therefore $u_{x_0}^{w}$ equals zero if $x$ solves one of the following equations: 
	\[  \frac{a_{+} x - d}{a_{+}t +b_{+}}  = -a_1, \quad 
		\frac{a_{+} x - d}{a_{+}t +b_{+}}  = a_2.
	\]
Hence the boundary conditions $u_{x_0}^w (-a_1 t -b_1, t) = u_{x_0}^w (a_2 t +b_2, t) =0$ are satisfied. Now we check the initial condition $u_{x_0}^w(x,0)$:
\begin{eqnarray}
\nonumber
u_{x_0}^w(x,0) 
&=&  
\frac{a_{+}}{b_{+}} 
e^{\frac{(a_{+}x_0 + d)^2}{2a_{+} b_{+}} - \frac{(a_{+} x - d)^2}{2a_{+} (a_{+}t + b_{+})}}
u_{(a_{+}x_0-d) / b_{+}}^s \left( \frac{a_{+} x - d}{b_{+}}, \, 0\right)
\\ \nonumber 
&=& 
\frac{a_{+}}{b_{+}}
e^{\frac{(a_{+}x_0 + d)^2}{2a_{+} b_{+}} - \frac{(a_{+} x - d)^2}{2a_{+} (a_{+}t + b_{+})}}
\delta \left(\frac{a_{+} x - d}{b_{+}} - \frac{a_{+}x_0-d} {b_{+}} \right)
\\ \nonumber 
&=& 
\frac{a_{+}}{b_{+}} 
e^{\frac{(a_{+}x_0 + d)^2}{2a_{+} b_{+}} - \frac{(a_{+} x - d)^2}{2a_{+} (a_{+}t + b_{+})}}
\delta \left(\frac{a_{+}}{b_{+}} (x-x_0) \right)
\\ \nonumber 
&=& 
\frac{a_{+}}{b_{+}} 
\delta \left(\frac{a_{+}}{b_{+}} (x-x_0) \right)
\\ \nonumber
&=& 
\delta \left(x-x_0\right)
\end{eqnarray}
\end{proof}

\section{Explicit formulas for wedge probability $\mathbb{P}(T^w > T)$}
\label{sec:Wedge_prob}
\subsection{Some auxiliary functions and terms}
Let's define the following function $\Psi(\alpha, \beta)$:
\begin{eqnarray}
\label{eq:Psi_def}
\Psi(\alpha, \beta)  = \frac{1}{\sqrt{2\pi}} \int_{\alpha}^{\beta} e^{-x^2/2} dx.
\end{eqnarray}
and recall well known Error function $erf(x)$ and normal c.d.f. $\Phi(x)$:
\begin{equation}
erf(x) = \frac{2}{\sqrt{\pi}} \int_{0}^{x} e^{-t^2} dt, \quad 
\Phi(x) =\frac{1}{\sqrt{2\pi}} \int_{-\infty}^{x} e^{-t^2/2} dt. 
\end{equation} 
Function $\Psi(\alpha, \beta)$ can be represented in terms of $erf(x)$ and $\Phi(x)$
\begin{equation}
\label{eq:Psi_def_2}
\Psi(\alpha, \beta) = \Phi(\beta) - \Phi(\alpha), \quad \Psi(\alpha, \beta) = \frac{erf(\beta/\sqrt{2}) - erf(\alpha/\sqrt{2})}{2}.
\end{equation}
Let us note that function $\Psi(\alpha, \beta)$ also has the following properties:
\begin{equation}
\label{eq:Psi_properties}
\Psi(\alpha, \beta) = \Psi (-\beta ,- \alpha), \quad \Psi(-\infty, +\infty) = 1.
\end{equation}

\subsection{Explicit formulas for wedge probability}
We will use the following notations for wedge probabilities $\mathbb{P} (T^w > T)$ and $\mathbb{P} (T^w  = \infty)$: 
\begin{equation}
k(a_1, b_1; a_2, b_2 ;T) = \mathbb{P} (T^w > T), \quad k(a_1, b_1; a_2, b_2) = \mathbb{P} (T^w = \infty), 
\end{equation}

\begin{proposition}[Anderson]
Wedge probability $k(a_1, b_1; a_2,b_2; T)$ has the following representation:
	\begin{equation}
	\label{eq:Anderson}
	\begin{array}{l}
	k(a_1,b_1; a_2, b_2; T) = \Psi\left(\frac{-a_1 T - b_1}{\sqrt{T}}, \frac{a_2 T + b_2}{\sqrt{T}} \right)  
	\\ 
	-
	\sum_{n = 1} ^{\infty} e^{-2 A_n}
	\Psi\left(\frac{-a_1 T - b_1 +2b_1- 4n b_{+}}{\sqrt{T}}, \,
	\frac{a_2 T + b_2 +2b_1- 4 n b_{+}}{\sqrt{T}}\right)
	\\ 
	-
	\sum_{n = 1} ^{\infty} e^{-B_n}
	\Psi\left(\frac{-a_1 T - b_1 +2b_1 + 4(n-1) b_{+}}{\sqrt{T}}, \,
	\frac{a_2 T + b_2 +2b_1 + 4(n-1) b_{+}}{\sqrt{T}}\right)
	\\ 
	+
	\sum_{n = 1} ^{\infty} e^{-2 C_n}
	\Psi\left(\frac{-a_1 T - b_1 + 4n b_{+}}{\sqrt{T}}, \,
	\frac{a_2 T + b_2 + 4n b_{+}}{\sqrt{T}} \right)
	\\ 
	+
	 \sum_{n = 1} ^{\infty} e^{-2 D_n}
	\Psi\left(\frac{-a_1 T - b_1 - 4n b_{+}}{\sqrt{T}}, \, 
	\frac{a_2 T + b_2 - 4n b_{+}}{\sqrt{T}}\right) 
	\end{array}
	\end{equation}
Here $A_n, B_n, C_n, D_n$ are defined by the following formulas:
\begin{equation}
\label{eq:ABCD_def}
\begin{array}{c}
A_n = n^2 a_2b_2 + (n-1)^2 a_1 b_1 + n(n-1)(a_2b_1 + a_1b_2), \\
B_n = (n-1)^2 a_2b_2 + n^2 a_1 b_1 + n(n-1)(a_2b_1 + a_1b_2), \\
C_n = n^2 (a_1b_1 + a_2b_2) + n(n - 1) a_2 b_1 + n(n + 1)a_1b_2, \\
D_n = n^2 (a_1b_1 + a_2b_2) + n(n + 1) a_2 b_1 + n(n - 1)a_1b_2,
\end{array}
\end{equation}
and function $\Psi$ is defined in (\ref{eq:Psi_def}).
\end{proposition}
\begin{proof}
Employ identity
\[
\frac{1}{b_{+} / a_{+}} - \frac{1}{T + b_{+} / a_{+}} =
\frac{a_{+}^2}{b_{+} (a_{+} T + b_{+})}
\]
in the (\ref{eq:u_w_from_u_s}) and choose the first formula for $u_{x_0}(x,t)$ from (\ref{eq:u_s_CFS}) with $x_0 = 0$.
\begin{equation}
k(a_1, b_1; a_2, b_2 ;T) = 
\frac{e^{\frac{d^2}{2a_{+} b_{+}}}}{\sqrt{2 \pi T}} \int_{-a_1 T-b_1}^{a_2 T + b_2}
\sum_{n= -\infty}^{+\infty}
e^{-\frac{(a_{+} x -d)^2}{2a_{+} (a_{+} T + b_{+})}} 
\left[ 
e^{-E_1(n)} - e^{-E_2(n)}\right] dx.
\end{equation}
Here
\[
E_1(n) = \left(\frac{a_{+} x - d}{a_{+} T + b_{+}} + \frac{d}{b_{+}} + 4 n a_{+} \right)^2 
,\quad E_2(n) = \left(\frac{a_{+} x - d}{a_{+} T + b_{+}} - \frac{d}{b_{+}}+ 2a_1 + 4 n a_{+} \right)^2 
\frac{b_{+} (a_{+} T + b_{+})}{2 a_{+}^2 T};
\]
One can check that 
\[
E_1(n) + \frac{(a_{+} x -d)^2}{2a_{+} (a_{+} T + b_{+})} - \frac{d^2}{2a_{+} b_{+}}
= \frac{(x + 4 n b_{+})^2}{2 T} + 4nd + 8n^2a_{+} b_{+} 
\]
and 
\begin{eqnarray}
\nonumber
E_2(n) + \frac{(a_{+} x -d)^2}{2a_{+} (a_{+} T + b_{+})} - \frac{d^2}{2a_{+} b_{+}}
&=& \frac{(x + 2b_1 + 4 n b_{+})^2}{2 T}  + 8n^2a_{+}b_{+} + 4n(2 a_1 b_{+} - d) +  2a_1b_1
\\ \nonumber
&=& \frac{(x + 2b_1 + 4 n b_{+})^2}{2 T}  + 8n^2a_{+}b_{+} + 2n(a_1 b_2 +b_1a_2)  + (4n + 2) a_1b_1.
\end{eqnarray}

In the result we have
\begin{equation}
\begin{array}{c}
k(a_1, b_1; a_2, b_2 ;T) = 
\sum_{n= -\infty}^{+\infty}
e^{-2((2n)^2a_{+} b_{+} + 2nd)}
\Psi \left(\frac{-a_1 T-b_1 + 4 n b_{+}}{\sqrt{T}}, \frac{a_2 T+ b_2 + 4 n b_{+}}{\sqrt{T}} \right)
\\
-\sum_{n= -\infty}^{+\infty} e^{-2((2n)^2a_{+}b_{+} + n(a_1 b_2 +b_1a_2) + (2n+1) a_1b_1)}
\Psi \left(\frac{-a_1 T-b_1 + 4 n b_{+} + 2b_1}{\sqrt{T}}, \frac{a_2 T+ b_2+ 4 n b_{+} +2b_1}{\sqrt{T}} \right).
\end{array}
\end{equation}
Rearrangement of the first and second sum 
\[
\sum_{n = -\infty}^{+\infty} F(n) = \sum_{n = 1}^{+\infty} F(n) + \sum_{n = 1}^{+\infty} F(-n) + F(0),
\quad 
\sum_{n = -\infty}^{+\infty} F(n) = \sum_{n = 1}^{+\infty} F(n - 1) + \sum_{n = 1}^{+\infty} F(-n)
\]
yields formula (\ref{eq:Anderson}). It is easy to check that (see also \cite{Ycart})
\begin{equation}
\begin{array}{l}
A_n = 4n^2a_{+}b_{+} - n(a_1 b_2 +b_1a_2) - (2n - 1) a_1b_1, \\
B_n = 4(n-1)^2 a_{+}b_{+} + (n - 1)(a_1 b_2 +b_1a_2) + (2n + 1) a_1b_1, \\
C_n = 4 n^2 a_{+}b_{+} + 2 n d, \\
D_n = 4 n^2 a_{+}b_{+} - 2 n d.
\end{array}
\end{equation}
\end{proof}

\begin{remark}
	Doob' formula 
\[
	k(a_1, b_1; a_2, b_2) = 1 - \sum_{n = 1} ^{\infty} e^{-2A_n} + e^{-2B_n} -e^{-2C_n} -e^{-2D_n}. 
	\]
	can be easily derived from (\ref{eq:Anderson}) if we tends $T \rightarrow \infty$. Note that if $a_1$ and $a_2$ have opposite signs then $k(a_1, b_1; a_2, b_2; \infty) = 0$.  
\end{remark}
\begin{proposition}
	Formula (\ref{eq:Anderson}) has the following alternatives:
	\begin{equation}
	\label{eq:My_formula}
	\begin{array}{c}
	k(a_1, b_1; a_2, b_2;T) =
	i \sqrt{\frac{\pi}{2 a_{+} b_{+}}} e^{\frac{d^2}{2a_{+}b_{+}}} \sum_{n= -\infty}^{+\infty} 
	e^{-\frac{\pi^2 n^2}{8 a_{+} b_{+}} - \frac{ i  n \pi a_1}{2a_{+}}}
	\sin\left( \frac{\pi n b_1}{2b_{+}}\right)
	\times
	\\ 
	\times 	\Psi \left(
	\frac{i \pi n  - 2a_1(a_{+} T + b_{+}) }{2a_{+}\sqrt{T + b_{+}/a_{+}}} 
	,\frac{i \pi n  +  2a_2 (a_{+} T + b_{+}) }{2a_{+}\sqrt{T + b_{+}/a_{+}}}  
	\right)
	\end{array}
	\end{equation}
or 
	\begin{equation}
	\label{eq:my_formula_comp}
	\begin{array}{c}
	k(a_1,b_1;a_2,b_2; T) = 
	\frac{e^{\frac{d^2}{2a_{+} b_{+}}} \sqrt{ T  + b_{+} / a_{+}} }{\sqrt{a_{+}b_{+}}} 
	\sum_{n = 1} ^{\infty}
	e^{-\frac{\pi^2 n^2}{8 a_{+}^2} \left(\frac{1}{b_{+}/a_{+}} - \frac{1}{T + b_{+}/a_{+}} \right)}
	\sin\left( \frac{\pi n b_1}{2 b_{+}}\right) \times 
	\\
	\times
	\int_{0}^{2 a_{+}} 
	\sin\left( \frac{\pi n  \varphi}{2 a_{+}} \right)
	e^{-\frac{(\varphi - a_1)^2 (T + b_{+} / a_{+})}{2}}
	d\varphi.
	\end{array}
	\end{equation}
\end{proposition}
Here function $\Psi$ is defined in (\ref{eq:Psi_def}).
\begin{proof}
From Propositions \ref{prop:u_s} and \ref{prop:u_w} we have the following formula for $k(a_1, b_1; a_2,b_2;T)$:
\begin{equation}
\nonumber
\begin{array}{c}
k(a_1,b_1;a_2,b_2; T) = 
\frac{e^{\frac{d^2}{2a_{+} b_{+}}}}{\sqrt{b_{+}} \sqrt{a_{+} T + b_{+}} }
\sum_{n = 1} ^{\infty}
e^{-\frac{\pi^2 n^2}{8 a_{+}^2} \left(\frac{1}{b_{+}/a_{+}} - \frac{1}{T + b_{+}/a_{+}} \right)}
\sin\left( \frac{\pi n b_1}{2 b_{+}}\right) 
\\
\int_{-a_1 T -b_1}^{a_2 T +b_2} 
\sin\left( \frac{\pi n }{ 2 a_{+}} \left( \frac{a_{+}x-d }{a_{+}T+b_{+}} +a_1\right)\right)
e^{-\frac{(a_{+} x - d)^2}{2a_{+} (a_{+}T + b_{+})}}
dx.
\end{array}
\end{equation}
Change of integration variable $\varphi =a_1 +  (a_{+} x - d)/(a_{+} T + b_{+})$ yields formula (\ref{eq:my_formula_comp}). Definite integral from (\ref{eq:my_formula_comp}) is known (see \cite{AbraStegun}) and can be expressed in terms of Error functions (from complex argument): 
\begin{equation}
\begin{array}{c}
\nonumber
\int_{0}^{2 a_{+}} 
\sin\left( \frac{\pi n  \varphi}{2 a_{+}} \right)
e^{-\frac{(\varphi - a_1)^2 (T + b_{+} / a_{+})}{2}}
d\phi = 
-\frac{i \sqrt{\pi /2}}{2 \sqrt{T + b_{+}/ a_{+}}} e^{- \frac{\pi^2 n^2}{8 a_{+}^2 (T + b_{+}/a_{+})}}
\bigg(
\\ \nonumber
e^{-\frac{n \pi a_1}{2a_{+}}} 
\left(
erf\left( \frac{i \pi n }{2\sqrt{2}a_{+}\sqrt{T + b_{+}/a_{+}}} - 
a_1 \frac{\sqrt{T + b_{+} / a_{+}}}{\sqrt{2}}\right)
- 
erf\left( \frac{i \pi n }{2\sqrt{2}a_{+}\sqrt{T + b_{+}/a_{+}}}  
+a_2 \frac{\sqrt{T + b_{+} / a_{+}}}{\sqrt{2}}\right)
\right)
\\ \nonumber
-
e^{\frac{n \pi a_1}{2a_{+}}} 
\left(
erf\left( -\frac{i \pi n }{2\sqrt{2}a_{+}\sqrt{T + b_{+}/a_{+}}} - 
a_1 \frac{\sqrt{T + b_{+} / a_{+}}}{\sqrt{2}}\right)
- 
erf\left( -\frac{i \pi n }{2\sqrt{2}a_{+}\sqrt{T + b_{+}/a_{+}}}  
+a_2 \frac{\sqrt{T + b_{+} / a_{+}}}{\sqrt{2}}\right)
\right)
\bigg)
\end{array}
\end{equation}
or (see (\ref{eq:Psi_def_2}))
\begin{equation}
\begin{array}{c}
\nonumber
\int_{0}^{2 a_{+}} 
\sin\left( \frac{\pi n  \varphi}{2 a_{+}} \right)
e^{-\frac{(\varphi - a_1)^2 (T + b_{+} / a_{+})}{2}}
d\phi = 
\frac{i \sqrt{\pi /2}}{2 \sqrt{T + b_{+}/ a_{+}}} e^{- \frac{\pi^2 n^2}{8 a_{+}^2 (T + b_{+}/a_{+})}}
\bigg(
\\ \nonumber
e^{-\frac{n \pi a_1}{2a_{+}}} 
\Psi\left( 
\frac{i \pi n }{2 a_{+}\sqrt{T + b_{+}/a_{+}}} - 
a_1 \sqrt{T + b_{+} / a_{+}}, 
\frac{i \pi n }{2 a_{+}\sqrt{T + b_{+}/a_{+}}}  
+a_2 \sqrt{T + b_{+} / a_{+}}
\right)
\\ \nonumber
-
e^{\frac{n \pi a_1}{2a_{+}}} 
\Psi\left( 
-\frac{i \pi n }{2 a_{+}\sqrt{T + b_{+}/a_{+}}} - 
a_1 \sqrt{T + b_{+} / a_{+}}, 
-\frac{i \pi n }{2 a_{+}\sqrt{T + b_{+}/a_{+}}}  
+a_2 \sqrt{T + b_{+} / a_{+}}
\right)
\bigg)
\end{array}
\end{equation}
Using the last formula in (\ref{eq:my_formula_comp}) and changing the sign of $n$ in the second term gives the representation (\ref{eq:My_formula}). 
\end{proof}

\begin{remark}
Ycart and Drouilhet formula can be derived from (\ref{eq:My_formula}) by $T \rightarrow \infty$:
\begin{eqnarray}
\nonumber
k(a_1, b_1; a_2, b_2;\infty) &=& 
	i \sqrt{\frac{\pi}{2 a_{+} b_{+}}}  e^{\frac{d^2}{2a_{+}b_{+}}} \sum_{n= -\infty}^{+\infty} 
\sin\left( \frac{\pi n b_1}{2b_{+}}\right)
e^{-\frac{\pi^2 n^2}{8 a_{+} b_{+}} - \frac{ i \pi a_1}{2a_{+}}}
\\ \nonumber
&=&
i \sqrt{\frac{\pi}{2 a_{+} b_{+}}} e^{\frac{d^2}{2a_{+}b_{+}}}\sum_{n= 1}^{+\infty} 
\sin\left( \frac{\pi n b_1}{2b_{+}}\right)
e^{-\frac{\pi^2 n^2}{8 a_{+} b_{+}}}
\left(e^{- \frac{ i \pi a_1}{2a_{+}}} - e^{\frac{ i \pi a_1}{2a_{+}}} \right).
\\ \nonumber
&=&
2\sqrt{\frac{\pi}{2 a_{+} b_{+}}} e^{\frac{d^2}{2a_{+}b_{+}}} \sum_{n= 1}^{+\infty} 
\sin\left( \frac{\pi n b_1}{2b_{+}}\right)
\sin \left(\frac{\pi n a_1}{2a_{+}}\right)
e^{-\frac{\pi^2 n^2}{8 a_{+} b_{+}}}.
\\ \nonumber
&=&
\sqrt{\frac{\pi}{2 a_{+} b_{+}}} e^{\frac{d^2}{2a_{+}b_{+}}} \sum_{n= 1}^{+\infty} 
\left(
\cos\left(\frac{\pi n b_1}{2b_{+}} - \frac{\pi n a_1}{2a_{+}}\right) -
\cos\left(\frac{\pi n b_1}{2b_{+}} + \frac{\pi n a_1}{2a_{+}}\right)
\right) 
e^{-\frac{\pi^2 n^2}{8 a_{+} b_{+}}}.
\\ \nonumber
&=&
\sqrt{\frac{\pi}{2 a_{+} b_{+}}} e^{\frac{d^2}{2a_{+}b_{+}}} \sum_{n= 1}^{+\infty} 
\left(
\cos\left(\frac{\pi n d}{2a_{+} b_{+}}\right) +
(-1)^{n+1} \cos\left(\frac{\pi n c}{2a_{+}b_{+}}\right)
\right) 
e^{-\frac{\pi^2 n^2}{8 a_{+} b_{+}}}.
\end{eqnarray}
Here $c$ is defined as $(a_1b_1 - a_2b_2)/2$ (see \cite{Ycart}).
\end{remark}

\section{Computational aspects}
\label{sec:numerics}
\subsection{Convergence rates} 
\begin{proposition}
Let's denote by $K_{1, N}$ and $K_{2,N}$ the partial sums in series (\ref{eq:Anderson}) and (\ref{eq:my_formula_comp}) respectively and consider the remainders:
\[
R_{1,N} = K_{1,\infty} - K_{1,N}, \quad R_{2, N} = K_{2,\infty} - K_{2,N}. 
\]
For $N > 1$:
\begin{eqnarray}
\label{eq:R_1_bound}
R_{1,N} &\leq&  \left\{ \begin{array}{l} 
 \sqrt{\frac{2}{\pi b_{+}^2 / T}}\frac{a_{+}b_{+} + b_{+}^2 / T}{ 
	 (4a_{+}b_{+} + b_{+}^2 / T)(N - 1)} e^{-\left(N - 1\right)^2 (4a_{+}b_{+} + 4b_{+}^2 /T)}, \quad 
N \geq a_{+} T / b_{+}.
\\
\frac{1}{4(N-1)a_{+}b_{+}} e^{-8\left(N - 1\right)^2 a_{+} b_{+} }.
\end{array} \right.
\\
\label{eq:R_2_bound}
R_{2, N} &\leq&  
2 \left( \frac{2}{\pi}\right)^{3/2}
\frac{a_{+} b_{+} + b_{+}^2 / T }{ N \sqrt{a_{+} b_{+}}} e^{2 a_{+} b_{+}} 
e^{-\frac{\pi^2 N^2}{8 (a_{+}b_{+} + b_{+}^2 / T)}}.
\end{eqnarray}
\end{proposition}
\begin{proof}
At first we note that $A_n$, $B_n$, $C_n$, $D_n$ are larger than $4(n-1)^2a_{+}b_{+}$ (see \cite{Ycart}) and apply second property of function $\Psi$ from (\ref{eq:Psi_properties}):
\begin{equation}
\begin{array}{l}
R_{1,N} \leq  \sum_{n = N + 1} ^{\infty} e^{-8(n-1)^2 a_{+} b_{+}}
\Psi\left(\frac{-a_1 T - b_1 +2b_1 + 4(n-1) b_{+}}{\sqrt{T}}, \,
          \frac{a_2 T + b_2 +2b_1 + 4(n-1) b_{+}}{\sqrt{T}}\right)
\\
+ 
\sum_{n = N + 1} ^{\infty} e^{-8(n-1)^2 a_{+} b_{+}}
\Psi\left(\frac{-a_2 T - b_2 -2b_1 + 4 n b_{+}}{\sqrt{T}}, \,
          \frac{a_1 T + b_1 -2b_1 + 4n b_{+}}{\sqrt{T}}
\right)
\\ \nonumber
+
\sum_{n = N + 1} ^{\infty} e^{-8(n-1)^2 a_{+} b_{+}}
\Psi\left(\frac{-a_1 T - b_1 + 4n b_{+}}{\sqrt{T}}, \,
          \frac{a_2 T + b_2 + 4n b_{+}}{\sqrt{T}} \right)
\\ 
+
\sum_{n = N + 1} ^{\infty} e^{-8(n-1)^2 a_{+} b_{+}}
\Psi\left(\frac{-a_2 T - b_2 + 4n b_{+}}{\sqrt{T}}, \,
          \frac{a_1 T + b_1 + 4n b_{+}}{\sqrt{T}} 
\right) 
\end{array}
\end{equation}
If quantity $a_{+} T / b_{+}$ is sufficiently large, then we can use Ycart and Drouilhet bounds \cite{Ycart} (this is second line in formula (\ref{eq:R_1_bound})) for the remainder $R_{1,N}$. If this term is sufficiently small, then we derive upper bound in the assumption $N > a_{+} T / b_{+}$.  

Now we show that for any $n \geq N+1$ the arguments of any function $\Psi(\alpha, \beta)$ from the formula above are positive: 
\begin{eqnarray}
\nonumber
-a_1 T - b_1 +2b_1 + 4(n-1) b_{+} 
&\geq& -2a_{+} T + b_1 + 4(n - 1) b_{+},
\\ \nonumber
&\geq&  -2a_{+} T + b_1 + 4 N b_{+}
\\ \nonumber
&\geq& 2N b_{+}
\end{eqnarray}
for the second and third line: 
\begin{eqnarray}
\nonumber
-a_2 T - b_2+ 4 n b_{+}
&\geq& -a_2 T - b_2 -2b_1 + 4 n b_{+}
\\ \nonumber
&\geq& -2a_{+} T - b_2 -2b_1 + 4 n b_{+}
\\ \nonumber
&\geq& -2a_{+} T - b_2 -2b_1 + 4 (N + 1) b_{+}
\\ \nonumber
&\geq& - b_2 -2b_1 + 4b_{+} +2 N b_{+}
\\ \nonumber
&\geq& 2 N b_{+}
\end{eqnarray}
and for fourth line we apply the following inequality:
\begin{eqnarray}
\nonumber
-a_1 T - b_1 + 4n b_{+}
\geq 
-2a_{+} T - b_2 -2b_1 + 4 n b_{+}.
\end{eqnarray}
Therefore, 
\begin{equation}
\nonumber
R_{1,N} \leq  4\sum_{n = N + 1} ^{\infty} e^{-8(n-1)^2 a_{+} b_{+}}
\Psi\left(\frac{2 n b_{+}}{\sqrt{T}}, \,
\frac{2a_{+} T + 2 b_{+} + 2 n b_{+}}{\sqrt{T}}\right)
\end{equation}
Application of the following inequality 
\[
\Psi(\alpha, \beta) \leq \frac{\beta - \alpha}{\sqrt{2\pi}} e^{-\alpha^2 / 2}, \quad \beta > \alpha > 0. 
\]
yields the following upper bound for the remainder $R_{1,N}$ :
\begin{eqnarray}
\nonumber
R_{1,N} &\leq&  \frac{8(a_{+} T + b_{+})}{\sqrt{2 \pi T}}
\sum_{n = N + 1} ^{\infty} e^{-8(n-1)^2 a_{+} b_{+} - 2 n^2 b_{+}^2 / T}
\\ \nonumber
&\leq&  \frac{8(a_{+} T + b_{+})}{\sqrt{2 \pi T}}
\sum_{n = N + 1} ^{\infty} e^{-(n-1)^2 (8a_{+} b_{+} + 2 b_{+}^2 / T) }, \quad 
\left \{ \sum_{n= N+1}^{\infty} e^{-un^2} \leq \frac{e^{-uN^2}}{2 u N}, 
\quad u = 8 a_{+} b_{+} + 2 b_{+}^2 / T \right\}
\\ \nonumber
&\leq&  \sqrt{\frac{2}{\pi T}} \frac{a_{+} T + b_{+}}{4a_{+}b_{+} + b_{+}^2 / T}
\frac{e^{-(N-1)^2(8 a_{+} b_{+} + 2 b_{+}^2 / T)}}{N - 1}
\end{eqnarray}
For the remainder $R_{2, N}$ we apply the following bounds:
\begin{eqnarray}
\nonumber
R_{2, N} &=&
\frac{e^{\frac{d^2}{2a_{+} b_{+}}} \sqrt{ T  + b_{+} / a_{+}} }{\sqrt{a_{+}b_{+}}} 
\sum_{n = N+1} ^{\infty}
e^{-\frac{\pi^2 n^2}{8 (a_{+} b_{+} + b_{+}^2 / T)}}
\sin\left( \frac{\pi n b_1}{2 b_{+}}\right)
\int_{0}^{2 a_{+}} 
\sin\left( \frac{\pi n  \varphi}{2 a_{+}} \right)
e^{-\frac{(\varphi - a_1)^2 (T + b_{+} / a_{+})}{2}}
d\varphi 
\\ \nonumber
&\leq&
\frac{e^{\frac{d^2}{2a_{+} b_{+}}}}{\sqrt{a_{+}b_{+}}} 
\sum_{n = N+1} ^{\infty}
e^{-\frac{\pi^2 n^2}{8 (a_{+} b_{+} + b_{+}^2 / T)}}
\sin\left( \frac{\pi n b_1}{2 b_{+}}\right)
\int_{\mathbb{R}} 
e^{-\varphi^2 / 2} d\varphi
\\ \nonumber
&\leq&
\sqrt{\frac{2\pi}{a_{+}b_{+}} } e^{\frac{d^2}{2a_{+} b_{+}}}
\sum_{n = N+1} ^{\infty}
e^{-\frac{\pi^2 n^2}{8 (a_{+} b_{+} + b_{+}^2 / T)}}, 
\quad \left \{ \sum_{n= N+1}^{\infty} e^{-un^2} \leq \frac{e^{-uN^2}}{2 u N}, \quad u = \frac{\pi^2 }{8 (a_{+}b_{+} + b_{+}^2 /T)} \right\}
\\ \nonumber
&\leq&
2 \left( \frac{2}{\pi}\right)^{3/2}
\frac{a_{+} b_{+} + b_{+}^2 / T }{ N \sqrt{a_{+} b_{+}}} e^{2 a_{+} b_{+}} 
e^{-\frac{\pi^2 N^2}{8 (a_{+}b_{+} + b_{+}^2 / T)}}.
\end{eqnarray}
\end{proof}
In case of $T = \infty$ the bounds for remainders depend only on $a_{+} b_{+}$. In the finite horizon case we have one extra term $b_{+}^2 / T$. We show that remainders $R_{1,N}$ and $R_{2, N}$ are characterized only by these terms. Denote by 
\begin{equation}
\label{eq:alpha_beta_def}
\alpha = a_{+} b_{+}, \quad \beta = b_{+}^2 / T
\end{equation}
and employ these new variables in formulas (\ref{eq:R_1_bound}) and (\ref{eq:R_2_bound}):
\begin{eqnarray}
R_{1,N}  &\leq& 
\left \{
\begin{array}{l}
  \sqrt{\frac{2}{\pi \beta}} \frac{\alpha + \beta}{4\alpha + \beta} 
\frac{e^{-(N-1)^2(8 \alpha + 2 \beta)}}{N - 1}, \quad N \geq \alpha / \beta.
\\
\frac{e^{-8 \alpha (N - 1)^2}}{4 \alpha (N - 1)}.
\end{array} \right.
\\
R_{2, N} &\leq&
2 \left( \frac{2}{\pi}\right)^{3/2}
\frac{\alpha + \beta }{ N \sqrt{\alpha} }e^{2\alpha} 
e^{-\frac{\pi^2 N^2}{8 (\alpha + \beta)}}.
\end{eqnarray}

\begin{figure}
	\begin{center}
		\resizebox*{18cm}{!}{\includegraphics{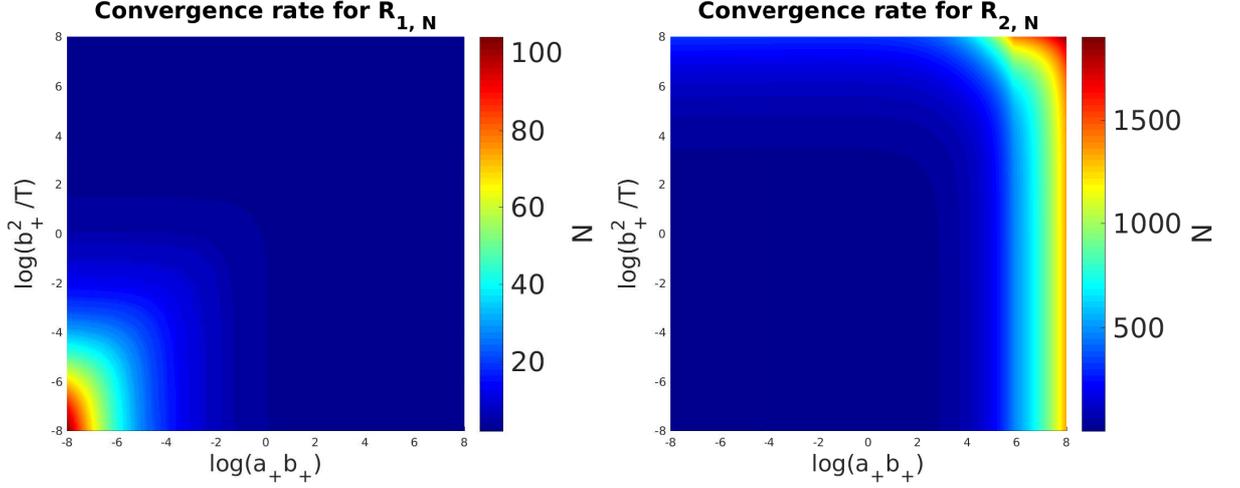}}
		\caption{\label{fig4} Minimal value of $N$ providing precision $10^{-16}$. Left: Anderson's formula. Right: Alternative.  
			\label{fig:conv_rate}}
	\end{center}
\end{figure}

\begin{figure}
	\begin{center}
		\resizebox*{18cm}{!}{\includegraphics{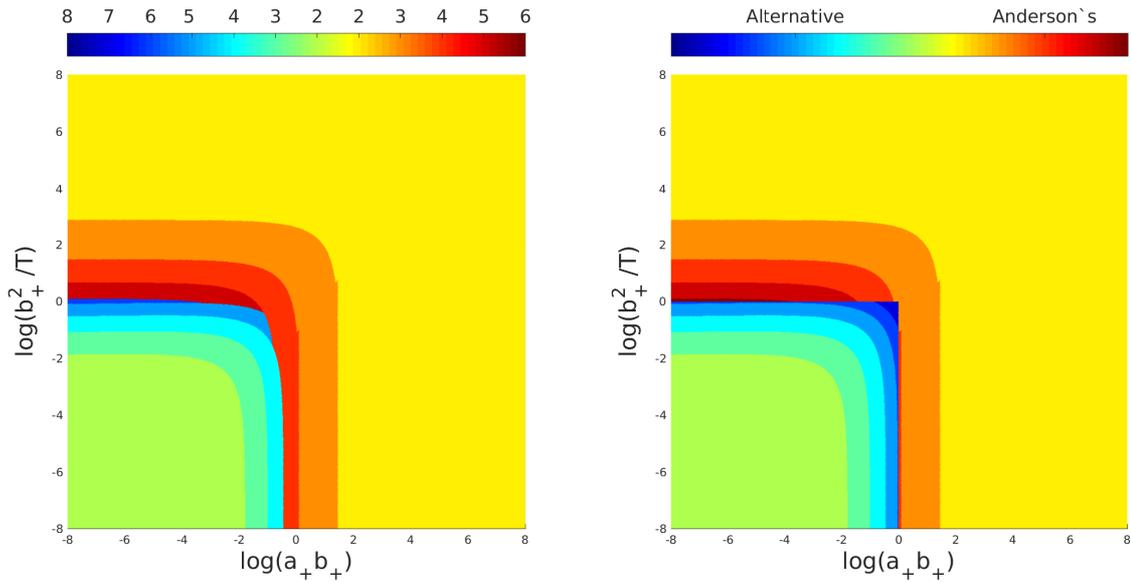}}
		\caption{\label{fig4} Minimal value of $N$ for Anderson's formula and alternative. Right: values for $N$ if we use alternative if and only if $\log(b_{+}^2 / T) < 0$ and $\log(a_{+}b_{+}) < 0$.   
			\label{fig:min_N}}
	\end{center}
\end{figure}

\subsection{Convergence analysis and implementation}
In this subsection we present numerical experiments illustrating convergence rates for the series (\ref{eq:Anderson}) and (\ref{eq:my_formula_comp}). As we mentioned before, convergence rates depend only on values of $\alpha$ and $\beta$ which are defined in (\ref{eq:alpha_beta_def}). Figure \ref{fig:conv_rate} illustrates the log-log plot of convergence rates. Color indicates minimal value of $N$ such that remainders $R_{1,N}$ and $R_{2,N}$ are less than $10^{-16}$. We can see that our formula convergences faster in case of $\alpha < 1$ and $\beta < 1$. In other cases Anderson's formula converges faster than alternative.  

Left sub-figure in \ref{fig:min_N} illustrates dependencies of minimal value $N$ on values of $\alpha$ and $\beta$. Let us mention, that minimal value of $N$ is equal to 5 for Anderson's formula and to 6 for the alternative. Formula (\ref{eq:My_formula}) should be used in the region defined by the curve which separates blue and red zones (see left Figure \ref{fig:min_N}). Implementation of this rule can be difficult, so we suggest the following rule: we use our formula if and only if $\log(\alpha) <0$ and $\log(\beta) <0$. Values of $N$ are presented in right Figure \ref{fig:min_N}. In this case minimal value of $N$ is increases up to 8 terms. 

We implement our algorithm in \textbf{C++11}. All terms in formula (\ref{eq:Anderson}) have been coded by using of functions from the standard library (\textbf{std::exp} and \textbf{std::erf}). Computation of definite integral in formula (\ref{eq:my_formula_comp}) have been implemented by the help of Simpson' formula. Numerical experiments have been made using vectors of simulated entries with uniform distribution on $[0, 10]$ for $a_1$, $b_1$, $a_2$, $b_2$ and on $[0, 100]$ for $T$. The running time on a standard laptop is approx. 7 seconds for $10^7$ values.
\begin{remark}
If we rearrange our formula (\ref{eq:my_formula_comp}) in the manner of Ycart and Drouilhet formula \cite[formula 4]{Ycart} (i.e. apply product-to-sum identity for sines), we shall cut $N$ in half. Hence, in these terms we have $N = 3$ (i.e. is equal to the infinite horizon case) for our formula.   
\end{remark}


\end{document}